\documentclass{amsart}

\usepackage{amssymb,amscd,pb-diagram}
\usepackage[all,arc]{xy}
\usepackage[english]{babel}
\usepackage{tabularx}

\newtheorem{theorem}{Theorem}[section]

\newtheorem{lemma}{Lemma}[section]
\newtheorem{corollary}{Corollary}[section]
\newtheorem{proposition}{Proposition}[section]
\newtheorem{definition}{Definition}[section]
\newtheorem{remark}{Remark}[section]
\newtheorem{example}{Example}[section]
\numberwithin{equation}{section}

\title{Holomorphic extension in holomorphic fiber bundles with (1,0)-compactifiable fiber}

\author{S.~V.~Feklistov}
\date{}

\begin{document}
\maketitle
\markright{The Hartogs phenomenon in holomorphic fiber bundles}
\begin{abstract}

We use the Leray spectral sequence for the sheaf cohomology groups with compact supports to obtain a vanishing result. The stalks of sheaves $R^{\bullet}\phi_{!}\mathcal{O}$ for the structure sheaf $\mathcal{O}$ on the total space of a holomorphic fiber bundle $\phi$ has canonical topology structures. Using the standard \v Cech argument we prove a density lemma for QDFS-topology on this stalks. In particular, we obtain a vanishing result for holomorphic fiber bundles with Stein fibers. Using K\"unnet formulas, properties of an inductive topology (with respect to the pair of spaces) on the stalks of the sheaf $R^{1}\phi_{!}\mathcal{O}$ and a cohomological criterion for the Hartogs phenomenon we obtain the main result on the Hartogs phenomenon for the total space of holomorphic fiber bundles with (1,0)-compactifiable fibers.
\end{abstract}

\section{Introduction}\label{sec1}

The classical Hartogs extension theorem states that for every domain $W\subset\mathbb{C}^{n}(n>1)$ and a compact set $K\subset W$ such that $W\setminus K$ is connected, the restriction homomorphism $$\mathcal{O}(W)\to \mathcal{O}(W\setminus K)$$ is an isomorphism.

A natural question arises if this is true for complex analytic spaces. 
\begin{definition}
We say that a noncompact connected complex analytic space $X$ admits the Hartogs phenomenon if for any domain $W\subset X$ and a compact set $K\subset W$ such that $W\setminus K$ is connected, the restriction homomorphism $$\mathcal{O}(W)\to \mathcal{O}(W\setminus K)$$ is an isomorphism. 
\end{definition}

In this or a similar formulation this phenomenon has been extensively studied in many situations, including Stein manifolds and spaces, $(n-1)$-complete normal complex spaces and so on \cite{Andersson,AndrHill,BanStan,ColtRupp, Fek1, Fek2, Harvey,Merker, Marc,Marc1, Vassiliadou,Viorel}.

Our goal is to study the Hartogs phenomenon in holomorphic fiber bundle with noncompact (1,0)-compactifiable fiber. Note that, a noncompact complex manifold $F$ is called $(1,0)$-com\-pa\-cti\-fiable (for more details, see Section \ref{sectioncompactif}) if it admits a compactification $F'$ with the following properties: 
\begin{enumerate}
\item $F'$ is a compact complex manifold;
\item $F'\setminus F$ is a proper connected analytic set;
\item $H^{1}(F',\mathcal{O}_{F'})=0$.
\end{enumerate}

Let $X,Y$ be a locally compact space and let $\phi\colon X\to Y$ be a continuous map. We use the Leray spectral sequence for a group cohomology with compact supports to obtain that $H^{i}_{c}(X,\mathcal{F})=0$ for all $i<q$ provided $R^{i}\phi_{!}\mathcal{F}=0$ for all $i<q$, where $\phi_{!}$ is the direct image with compact supports between sheaves categories $\mathcal{SH}(X)$ and $\mathcal{SH}(Y)$, $R^{i}\phi_{!}$ is the $i$-th derived functor (Corollary \ref{corcompactspectral}). 

Now let $\phi\colon X\to Y$ be a holomorphic fiber bundle with noncompact fiber $F$, and $y\in Y$ be a point. The stalk $(R^{i}\phi_{!}\mathcal{O}_{X})_{y}$ has the canonical topological vector space structure which is QDFS-space (see Section 3). Using the standard \v Cech argument we prove that the canonical homomorphism $$\mathcal{O}_{y}\otimes H^{i}_{c}(F_{y},\mathcal{O}_{F_y})\to (R^{i}\phi_{!}\mathcal{O}_{X})_{y}$$ has a dense image with respect to the QDFS-topology (Lemma \ref{Lemmaondense}). In particular, if $(R^{i}\phi_{!}\mathcal{O}_{X})_{y}$ is a separated space, then the condition $H^{i}_{c}(F,\mathcal{O}_{F})=0$ for all $i<q$ implies that $H^{i}_{c}(X,\mathcal{O}_{X})=0$ for all $i<q$ (Corollary \ref{vanishtheorem}). 

If the fiber $F$ is a Stein manifold, then $H^{i}_{c}(X,\mathcal{O}_{X})=0$ for all $i<\dim F$ (Corollary \ref{Stein}). For $F=\mathbb{C}^{n}$ and $n>1$ we obtain the Dwilewicz's result about vanishing $H_{c}^{1}(X,\mathcal{O}_{X})$ \cite[Corollary 1.4]{Dwilewicz}).

Now, assume that the fiber $F$ is $(1,0)$-compactifiable and $\dim F>1$. In this case, using the K\"unnet formulas for algebraic and topological tensor products, the long exact sequence for the pair $(F,F')$ and Lemma \ref{Lemmaondense} we obtain that the condition $H^{1}_{c}(F,\mathcal{O}_{F})=0$ implies that $R^{1}\phi_{!}\mathcal{O}_X=0$. In particular $H^{1}_{c}(X,\mathcal{O}_{X})=0$ provided $H^{1}_{c}(F,\mathcal{O}_{F})=0$ (Corollary \ref{maincorcompactif}). 
 
Using the cohomological criterion for the Hartogs phenomenon (see Section 6) we obtain the following main result (Theorem \ref{mainresult}).

\begin{theorem}
Let $\phi\colon X\to Y$ be a holomorphic fibre bundle with $(1,0)$-com\-pac\-ti\-fiable fiber $F$, $\dim F>1$. If $F$ admits the Hartogs phenomenon, then $X$ also admits the Hartogs phenomenon. 
\end{theorem}

For example, let $G$ be a semiabelian Lie group (i.e. G is an extension of an abelian manifold $A$ by an algebraic torus $T\cong (\mathbb{C}^{*})^{n}$). We have a principle $T$-bundle $G\to A$. Let $F$ be a toric $T$-manifold which has only one topological end (about toric varieties see, for instance, \cite{Oda}), $\dim F>1$, and $F$ admits the Hartogs phemomenon. Then the total space of an associated fiber bundle $G\times^{T}F\to A$ admits the Hartogs phenomenon (see Example \ref{Examtoric}).

\section{The Leray spectral sequence for sheaf cohomology groups with compact supports}

In this section we recall the Leray spectral sequence (see \cite[4.1.3]{Voisin}).

\subsection{The spectral sequence of composed functors}
Now we recall the spectral sequence of composed functors (see \cite[Section 4.1.2]{Voisin}).

Let $\mathcal{A},\mathcal{B},\mathcal{C}$ be abelian categories, where $\mathcal{A}$ and $\mathcal{B}$ have sufficiently many injective objects. Let $S\colon \mathcal{A}\to \mathcal{B}, S'\colon\mathcal{B}\to \mathcal{C}$ be left exact functors.  Assume that $S$ transform the injective objects of $\mathcal{A}$ into $S'$-acyclic objects of $\mathcal{B}$.

Then we have the following result (see \cite[Theorem 4.9]{Voisin}) which is based on the Cartan-Eilenberg resolution. 

\begin{theorem}\label{composedfunct}
There is a canonical filtration $F$ on the objects $R^{i}(S'\circ S)(M)$ and a spectral sequence $$E_{r}^{p,q}\Longrightarrow R^{p+q}(S'\circ S)(M)$$ with $E_{2}^{p,q}=R^{p}F'(R^{q}S(M))$, $E_{\infty}^{p,q}=Gr^{p}_{F}R^{p+q}(S'\circ S)(M)$.
\end{theorem}

\subsection{The Leray spectral sequenses}
Recall the classical Leray spectral sequence (see \cite[4.1.3]{Voisin}). It is a special case of spectral sequence of a composed functors. Let $X, Y$ be topological spaces and $\phi\colon X\to Y$ be a continuous map. Denote by $\mathcal{SH}(Z)$ the category of sheaves of abelian group on the topological manifold $Z$. Let $\mathcal{AB}$ be the category of abelian groups.

Consider the direct image functor $\phi_{*}\colon \mathcal{SH}(X) \to \mathcal{SH}(Y)$ and the global sections functor $\Gamma(Y,\bullet)\colon \mathcal{SH}(Y)\to \mathcal{AB}$. Note that, for any sheaf $\mathcal{F}\in \mathcal{SH}(X)$ we have $\Gamma(X,\mathcal{F})=\Gamma(Y,\phi_{*}\mathcal{F})$.

Since each injective sheaf on $X$ is a flabby, each flabby sheaf is $\Gamma$-acyclic and $\phi_{*}\mathcal{F}$ is flabby sheaf on $Y$ for every flabby sheaf $\mathcal{F}$ on $X$, it follows that for every injective sheaf $\mathcal{I}$ we have that $\phi_{*}\mathcal{I}$ is a $\Gamma$-acyclic sheaf on $Y$.

We obtain the following theorem (\cite[4.1.3]{Voisin}):

\begin{theorem}
For every sheaf $\mathcal{F}\in \mathcal{SH}(X)$, there exists a canonical filtration $F$ on $H^{q}(X,\mathcal{F})$, and a spectral sequence $E_{r}^{p,q}\Longrightarrow H^{p+q}(X,\mathcal{F})$ that is canonical starting from $E_2$, and satisfies $E_{2}^{p,q}=H^{p}(Y,R^{q}\phi_{*}\mathcal{F}), E^{p,q}_{\infty}=Gr_{F}^{p}H^{p+q}(X,\mathcal{F})$.
\end{theorem}

We have to study the group cohomology with compact supports of sheaf on $X$. In this case we consider the functors $\Gamma_{c}$ and $\phi_{!}$ instead $\Gamma$ and $\phi_{*}$.  

Let $X$ be a locally compact space. Recall that the support of a global section $s\in \Gamma(X,\mathcal{F})$ is the set $supp(s)=\{x\in X\mid s_{x}\neq 0\}$.

\begin{definition}
For a sheaf $\mathcal{F}\in\mathcal{SH}(X)$ put $$\Gamma_{c}(X,\mathcal{F}):=\{s\in \Gamma(X,\mathcal{F})\mid supp(s) \text{ is compact}\}.$$
\end{definition}

This defines a left exact functor $\Gamma_{c}(X,-)\colon \mathcal{SH}(X)\to\mathcal{AB}$. The i'th derived functor of $\Gamma_{c}(X,-)$ evaluated on the sheaf $\mathcal{F}$ will be denoted $H^{i}_{c}(X,\mathcal{F})$. 

\begin{remark}
Let $X$ be a locally compact space. Note that any soft sheaf $S\in\mathcal{SH}(X)$ is $\Gamma_{c}(X,-)$-acyclic \cite[Secton III, Theorem 2.7]{Iversen}. 
\end{remark}

Now let $\phi\colon X\to Y$ be a continuous map between locally compact spaces. Define the functor $\phi_{!}$ (see \cite[Sections III and VII]{Iversen}).

\begin{definition}
For a sheaf $\mathcal{F}\in \mathcal{SH}(X)$ and an open subset $V\subset Y$ put $$\Gamma(V,\phi_{!}\mathcal{F}):=\{s\in \Gamma(\phi^{-1}(V),\mathcal{F})\mid \phi\mid_{supp(s)}\colon supp(s)\to Y \text{ is a proper map}\}.$$
We consider $\phi_{!}\mathcal{F}$ as a subpreasheaf of $\phi_{*}\mathcal{F}$.
\end{definition}

\begin{remark}\label{directimage}
Now we list some properties of the presheaf $\phi_{!}\mathcal{F}$ (see \cite[Section VII]{Iversen}).
\begin{enumerate}
\item The presheaf $\phi_{!}\mathcal{F}$ is a sheaf on $Y$; In particular we obtain the left exact functor $\phi_{!}\colon \mathcal{SH}(X)\to \mathcal{SH}(Y)$ and it has the derived functors which will be denoted $R^{i}\phi_{!}$.
\item For $y\in Y$ we have a natural isomorphism $$(R^{i}\phi_{!}\mathcal{F})_{y}\cong H^{i}_{c}(\phi^{-1}(y),\mathcal{F}\mid_{\phi^{-1}(y)}).$$
\item  A soft sheaf $\mathcal{S}\in\mathcal{SH}(X)$ is 
transformed into a soft sheaf $\phi_{!}\mathcal{S}\in\mathcal{SH}(Y)$. Moreover $R^{i}\phi_{!}\mathcal{S}=0$ for $i>0$.
\end{enumerate}
\end{remark}

The following lemma allows to apply Theorem \ref{composedfunct}. 

\begin{lemma}\label{spectrallemma}
Let $X,Y$ be locally compact spaces and $\phi\colon X\to Y$ be a continuous map. 
\begin{enumerate}
\item For any sheaf $\mathcal{F}\in \mathcal{SH}(X)$ we have $\Gamma_{c}(Y,\phi_{!}\mathcal{F})=\Gamma_{c}(X,\mathcal{F})$;
\item If $\mathcal{I}\in\mathcal{SH}(X)$ is an injective sheaf, then $\phi_{!}\mathcal{I}$ is $\Gamma_{c}(Y,-)$-acyclic. 
\end{enumerate}
\end{lemma}
\begin{proof}
If $s\in \Gamma_{c}(X,\mathcal{F})$, then $\phi\mid_{supp(s)}\colon supp(s)\to Y$ is a proper map. It follows that $\Gamma_{c}(Y,\phi_{!}\mathcal{F})=\Gamma_{c}(X,\mathcal{F})$.

Futher, since each injective sheaf $\mathcal{I}$ is soft, it follows that $\phi_{!}\mathcal{I}$ is soft (by Remark \ref{directimage}). But each soft sheaf is $\Gamma_{c}(Y,-)$-acyclic.  
\end{proof}

So, we obtain the following result:

\begin{theorem}\label{compactcohomspectr}
Let $X,Y$ be locally compact spaces and let $\phi\colon X\to Y$ be a continuous map. For every sheaf $\mathcal{F}\in \mathcal{SH}(X)$, there exists a canonical filtration $F$ on $H_{c}^{q}(X,\mathcal{F})$, and a spectral sequence $$E_{r}^{p,q}\Longrightarrow H_{c}^{p+q}(X,\mathcal{F})$$ that is canonical starting from $E_2$, and satisfies $E_{2}^{p,q}=H^{p}_{c}(Y,R^{q}\phi_{!}\mathcal{F}), E^{p,q}_{\infty}=Gr_{F}^{p}H_{c}^{p+q}(X,\mathcal{F})$.
\end{theorem}

\begin{proof}
We apply Theorem \ref{composedfunct} to $\mathcal{A}=\mathcal{SH}(X),\mathcal{B}=\mathcal{SH}(Y),\mathcal{C}=\mathcal{AB}$, $F=\phi_{!}, F'=\Gamma_{c}(Y,-)$ and use Lemma \ref{spectrallemma}.
\end{proof}

In particalar we obtain the following corollary

\begin{corollary}\label{corcompactspectral}
Let $X,Y$ be locally compact spaces and let $\phi\colon X\to Y$ be a continuous map and $\mathcal{F}\in \mathcal{SH}(X)$. If $R^{i}\phi_{!}\mathcal{F}=0$ for all $i<q$, then $H^{i}_{c}(X,\mathcal{F})=0$ for all $i<q$ and $H^{q}_{c}(X,\mathcal{F})\cong \Gamma_{c}(Y,R^{q}\phi_{!}\mathcal{F})$. 
\end{corollary}
\begin{proof}
Note that, $F^{s}H^{i}(X,\mathcal{F})=0$ for all $s>i$. Since $E_{\infty}^{k,l}=0$ for all $k+l=i<q$, it follows that $H^{i}_{c}(X,\mathcal{F})=0$ for all $i<q$. 

Futher, we have $E_{2}^{p,i}=0$ for all $p$ and for all $i<q$. It follows that $E^{i,q-i}_{\infty}=Gr^{i}_{F}H^{q}_{c}(X,\mathcal{F})=0$ for all $0<i\leq q$. This means that $F^{i}H^{q}_{c}(X,\mathcal{F})=F^{i+1}H^{q}_{c}(X,\mathcal{F})=0$ for all $1\leq i\leq q$. 

Since the differentials $d_{r}\colon E_{r}^{0,q}\to E_{r}^{r,q-r+1}=0$ and $0=E_{r}^{-r,q+r-1}\to E_{r}^{0,q}$ are zero, it follows that $H^{q}_{c}(X,\mathcal{F})=Gr_{F}^{0}H^{q}_{c}(X,\mathcal{F})=E_{\infty}^{0,q}\cong E_{2}^{0,q}=\Gamma_{c}(Y,R^{q}\phi_{!}\mathcal{F})$. 
\end{proof}

\subsection{Differential-geometric description of $R^{i}\phi_{!}\mathcal{O}_{X}$}

For the structure sheaf $\mathcal{O}_{X}$ on the total space $X$ of a holomorphic fiber bundle $\phi\colon X\to Y$ with a fiber $F$ we have a more explicit formula for the sheaves $R^{i}\phi_{!}\mathcal{O}_{X}$. For the sheaf $\mathcal{O}_{X}$ there exists the Dolbeault resolution $(\mathcal{A}_{X}^{0,\bullet},\bar{\partial})$. Since the sheaves $\mathcal{A}_{X}^{0,\bullet}$ are c-soft, then the resolution $(\mathcal{A}_{X}^{0,\bullet},\bar{\partial})$ is $\phi_{!}$-injective. It follows that $$R^{i}\phi_{!}\mathcal{O}_{X}\cong \frac{Ker (\bar{\partial}\colon\phi_{!}\mathcal{A}_{X}^{0,i}\to \phi_{!}\mathcal{A}_{X}^{0,i+1})}{Im (\bar{\partial}\colon\phi_{!}\mathcal{A}_{X}^{0,i-1}\to \phi_{!}\mathcal{A}_{X}^{0,p})}.$$

For any section $\omega\in R^{i}\phi_{!}\mathcal{O}_{X}(U)$ there exists a covering $U=\bigcup\limits_{\alpha\in A}U_{\alpha}$ and $\omega_{\alpha}\in \mathcal{A}^{0,i}(U_{\alpha}\times F)$ such that 
\begin{enumerate}
\item $\bar{\partial}\omega_{\alpha}=0$; 
\item $pr_{1}\colon supp(\omega_{\alpha})\to U_{\alpha}$ is a proper map;
\item For any $\alpha,\beta$ there exists a covering $U_{\alpha}\cap U_{\beta}=\bigcup\limits_{i}U_{\alpha\beta, i}$ and $$\eta_{\alpha\beta,i}\in \mathcal{A}^{0,i-1}(U_{\alpha\beta, i}\times F)$$ such that $supp(s_{i})\to U_{\alpha\beta, i}$ is a proper map and $$\bar{\partial}\eta_{\alpha\beta,i}=(\omega_{\alpha}-\omega_{\beta})\mid_{U_{\alpha\beta,i}}.$$
\end{enumerate}

But this description of the sheaves $R^{i}\phi_{!}\mathcal{O}_{X}$ is not used in this paper.

\section{Canonical topologies on the sheaf cohomology groups and the topological K\"unnet formula}

In this section we recall some notions and facts from topological vector spaces (for more details see \cite{GROTHENDIECK, Schaefer, Pietsch}) and its applications to the cohomology theory of sheaves (see \cite{BanStan}). 

\subsection{FS and DFS type spaces and properties}
A locally convex topological vector space over $\mathbb{C}$, in addition metrizable and complete is called a Fr\'echet space. A Fr\'echet space is called FS space (Fr\'echet-Schwartz space) is for any neighbourhood of zero $U$ there is a neighbourhood of zero $V$ such that: for any $\varepsilon>0$ there are $x_{1},\cdots, x_n\in V$ with the property $V \subset \bigcup\limits_{i=1}^{n}(x_{i}+\varepsilon U)$. A locally convex topological vector space is called a DFS space if it is the strong dual of FS space. If $E$ is a localy convex space, then we denote by $E'$ the strong dual space of $E$. By QFS and QDFS we mean quotients of such spaces. 

Let $f\colon E\to F$ be a continuous $\mathbb{C}$-linear map between topological vector spaces over $\mathbb{C}$. The map $f$ is called a topological homomorphism if the quotient topology on $f(E)$ coincides with the induced topology from $F$. 

Two topological vector spaces $E, F$ are in topological duality if there exists a bilinear and separately continuous map $E\times F\to\mathbb{C}$ such that the induced maps $E \to F'$, $F \to E'$ are topological isomorphisms. 

A complex $(E^{\bullet},d)$ is called complexes of topological vector spaces if every $E^{i}$ is a topological vector space over $\mathbb{C}$ and $d^{i}\colon E^{i}\to E^{i+1}$ is a continuous $\mathbb{C}$-linear map. A continuous morphism between two TVS complexes is a morphism of complexes (of zero degree) whose components are continuous and $\mathbb{C}$-linear. If $f\colon E^{\bullet} \to F^{\bullet}$ is such a morphism, then the linear maps $$H^{n}(f)\colon H^{n}(E^{\bullet})\to H^{n}(F^{\bullet})$$ are continuous.

Now we collect some facts about FS and DFS spaces(for details see \cite{BanStan}).
\begin{itemize}
\item FS(DFS) space is reflexive;
\item Any closed subspace of FS (DFS) space is an FS (DFS) space;
\item Quotient of an FS (DFS) space by a closed subspace is also an FS (DFS) space;
\item A surjective continuous $\mathbb{C}$-linear map between two FS (DFS) spaces is open; 
\item Any separated locally convex topological vector space, countable inductive limit of Fr\'echet spaces such that the maps of the inductive system are compact, is DFS;
\item Any locally convex topological vector space, countable inductive limit of Fr\'echet spaces such that the maps of the inductive system are compact and injective, is separated and DFS;
\item If $E=\varinjlim E_n$ is an inductive limit (in the category of locally convex topological vector spaces) of FS spaces such that the maps $E_n \to E_{n+1}$ are compact and
in addition E is separated, then any bounded subset of $E$ is the image of a bounded subset of some $E_n$;
\item Let $f\colon E\to F$ be a continuous linear map between FS (DFS) spaces, then $f$ is a topological homomorphism if and only if $f(E)$ is a closed subspace $F$;
\item If $E\overset{f}{\to} F\overset{g}{\to} G$ is a complex of FS (DFS) spaces, then the cohomology space $ker g/im f$ is separated if and only if $f$ is a topological homomorphism;
\item Let $E\overset{f}{\to} F\overset{g}{\to} G$ be a complex between locally convex spaces and $g$ is a topological homomorphism, let $G'\overset{g'}{\to} F'\overset{f'}{\to} E'$ be the corresponding topological dual complex. Then there is a natural algebraic isomorphism $$Ker f'/Im g'\cong (Ker g/Im f)';$$
\item Let $f\colon E^{\bullet}\to F^{\bullet}$ be a continuous morphism between two FS (DFS) complexes. If $f$ is an algebraic quasi-isomorphism, then it is a topological quasi-isomorphism (induces topological isomorphism to 
cohomology).
\item Let $E^{\bullet}$ be an acyclic complex of FS or DFS spaces. Then the dual complex $E'^{\bullet}$ is acyclic.
\item Let $f\colon E^{\bullet}\to F^{\bullet}$ be a topological quasi-isomorphism between  FS or DFS complexes. Then the dual morphism $f'$ is a topological quasi-isomorphism.
\end{itemize}

Now we consider some classical examples. 
\begin{example}
Let $X$ be a complex manifold, $\dim X=n$. Let $\mathcal{A}_{X}^{p,q}$ ($\mathcal{D}^{p,q}_{X}$) be the sheaf of germs of differential forms of the type $(p,q)$ with coefficients $C^{\infty}$-functions (distributions respectively), $\mathcal{O}_{X}$ be the sheaf of germs of holomorphic functions. 
\begin{enumerate}
\item The space of global section $\Gamma(X, \mathcal{A}_{X}^{p,q})$ is an FS space (with respect to the topology of the 
uniform convergence of the coefficients of the form and all their derivatives); 
\item The space of global section with compact support $\Gamma_{c}(X, \mathcal{A}^{p,q}_{X})$ is a DFS space (since $ \Gamma_{c}(X, \mathcal{A}^{p,q}_{X})=(\Gamma(X, \mathcal{A}_{X}^{n-p,n-q}))'$);
\item The space $\Gamma(X,\mathcal{O}_{X})$ is an FS space (with respect to the topology of the uniform convergence on 
compacts);
\item If $U$ is a relatively compact open subset of $X$, then the restriction map $\mathcal{O}(X) \to \mathcal{O}(U)$ is compact. Let $K$ be a compact subset of $X$ and $\{U_{n}\}_{n\geq 0}$ a fundamental system of neighbourhoods of $K$
such that $U_{n+1}\Subset U_{n}$ for any $n>0$. Suppose in addition that any connected component of each $U_n$ does intersect $K$. The restriction maps $\mathcal{O}(U_{n})\to\mathcal{O}(U_{n+1})$ are compact and injective. It follows that the space $\mathcal{O}(K)=\varinjlim \mathcal{O}(U_n)$ (with topology of inductive limit in the category of locally convex topological vector space) is separated, a DFS and LF space. And any bounded subset of $\mathcal{O}(K)$ is the image of a bounded subset of some $\mathcal{O}(U_n)$.
\item The spaces $H^{\bullet}_{c}(X,\mathcal{O}_{X})$ has a natural QDFS topology, since it can be computed by means of the complex $\Gamma_{c}(X,\mathcal{D}^{0,\bullet}_{X})$ and  $\Gamma_{c}(X,\mathcal{D}^{0,\bullet}_{X})=(\Gamma(X,\mathcal{E}^{n,n-\bullet}))'$.
\end{enumerate}
\end{example}

Let $X$ be a complex space and $\mathcal{F}$ be a coherent analytic sheaf on $X$. The space of global section $\Gamma(X,\mathcal{F})$ has natural structure of a topological vector space, which is FS if X has a countable topology (a countable basis of open subsets) \cite{Ganning}. Thus $\mathcal{F}$ becomes a Fr\'echet-Schwartz sheaf, that is a sheaf such that for any open subset $U$, $\Gamma(U, \mathcal{F})$ has an FS structure and for any two open subsets $V\subset U$, the restriction $\Gamma(U, \mathcal{F})\to \Gamma(V, \mathcal{F})$ are continuous. 

If $\mathcal{F}\to\mathcal{G}$ is a morphism of coherent analytic sheaves on $X$, then the map $\Gamma(X, \mathcal{F})\to \Gamma(X, \mathcal{G})$ is continuous. Moreover, if $X$ is a Stein space, then the map $\Gamma(X, \mathcal{F})\to \Gamma(X, \mathcal{G})$ is a topological homomorphism. 

\begin{remark}\cite[Chapter I, Proposition 2.10 and Lemma 2.11]{BanStan}\label{propDFSstein}
Let $X$ be a complex space, $K\subset X$ a Stein compact. Let  $\mathcal{F}, \mathcal{G}$ be coherent sheaves on $X$. Then 
\begin{enumerate}
\item the topology of inductive limit (in the category of local convex topological vector spaces) on $\Gamma(K,\mathcal{F})$ is separated, LF and DFS; 
\item the map $\Gamma(K,\mathcal{F})\to\Gamma(K,\mathcal{G}) $ is a topological homomorphism.
\end{enumerate}
\end{remark}

\subsection{QFS-topology on $H^{\bullet}(X,\mathcal{F})$}
Let $X$ be a complex space with countable topology, $\mathcal{F}$ be a coherent sheaf on $X$. Let $\mathfrak{U}$ be a countable Stein open covering of $X$. We have a canonical isomorphism $H^{\bullet}(\mathfrak{U},\mathcal{F})\cong H^{\bullet}(X,\mathcal{F})$. Let $U\in\mathfrak{U}$; then $\Gamma(U,\mathcal{F})$ has the FS topology, then the \v Cech complex $$C^{\bullet}(\mathfrak{U},\mathcal{F})=\prod\limits_{s=(s_{0},\cdots,s_p)} \Gamma(U_{s_0}\cap\cdots\cap U_{s_p},\mathcal{F})$$ becomes an FS complex. So, by passing to cohomology, one gets the natural QFS topology on $H^{\bullet}(X,\mathcal{F})$ (which is independent on the covering $\mathfrak{U}$). If $\mathcal{F}\to \mathcal{G}$ is a morphism of coherent sheaves, then the maps $H^{\bullet}(X, \mathcal{F})\to H^{\bullet}(X,\mathcal{G})$ are continuous. Moreover, if $H^{q}(X, \mathcal{F})\to H^{q}(X,\mathcal{G})$ is surjective, then it is topological. 

If we have an exact sequence $$0\to\mathcal{F}'\to\mathcal{F}\to\mathcal{F}''\to 0$$ then the connection homomorphisms $\delta\colon H^{q}(X,\mathcal{F}'')\to H^{q+1}(X,\mathcal{F}')$ are continuous. 

\subsection{QDFS-topology on $H^{\bullet}_{c}(X,\mathcal{F})$}
We now consider topologies on the groups $H^{\bullet}_{c}(X,\mathcal{F})$. For any Stein compact $K$ the space $\Gamma(K,\mathcal{F})$ has a natural $LF$ structure given by the equality $$\Gamma(K,\mathcal{F})=\varinjlim\limits_{U\supset K}\Gamma(U,\mathcal{F}).$$ This topological structure is DFS. 

Let $\mathfrak{K}$ be a locally finite covering of $X$ by Stein compacts (compact subsets of a complex space, admitting a fundamental system of Stein neighbourhoods). 

Note that $H^{q}(K,\mathcal{F})=0$ for $q\geq 1$ and for any Stein compact $K$. Define the complex of finite cochains $$C_{c}^{p}(\mathfrak{K},\mathcal{F}):=\bigoplus\limits_{s=(s_{0},\cdots,s_p)} \Gamma(K_{s_0}\cap\cdots\cap K_{s_p},\mathcal{F}).$$ 

Each $C^{p}_{c}(\mathfrak{K},\mathcal{F})$ has a natural DFS structure and one can easily show that the differentials $C^{p}_{c}(\mathfrak{K},\mathcal{F})\to C^{p+1}_{c}(\mathfrak{K},\mathcal{F})$ are continuous. Then on the groups $H^{\bullet}_{c}(X,\mathcal{F})$ a QDFS topology is obtained (which is independent of the covering) and it is called the natural QDFS topology.

\begin{remark}\cite[Chapter VII, Proposition 4.5]{BanStan}
Let $X$ be a complex manifold with countable basis and $\mathcal{F}$ be a coherent analytic sheaf on $X$, then the natural topologies on $H^{\bullet}(X,\mathcal{F})$ and $H^{\bullet}_{c}(X,\mathcal{F})$ obtained by means of the \v Cech cohomology, coincide with the topologies obtained by means of the Dolbeault-Grothendieck resolutions.
\end{remark}

\subsection{QDFS-topology on $H^{\bullet}_{c}(Z_y,\mathcal{F}\mid_{Z_y})$}

Now, let $X=Y\times F$ be the product of complex manifolds with countable basis, $\mathcal{F}$ be a coherent sheaf on $X$. Let $Z\subset F$ be a closed complex submanifold, $F_{y}=F\times \{y\}$, $Z_{y}=Z\times \{y\}$. Now we define the natural QDFS-topology on $H^{\bullet}_{c}(Z_y,\mathcal{F}\mid_{Z_y})$.

Let $\mathfrak{K}$ be a locally finite covering by Stein compacts of $Z_{y}$. Note that by Siu's theorem (see \cite[Theorem 3.1.1]{Forst}) any Stein compact $K\subset Z_y$ has a countable system of Stein neighbourhoods of the form $U\times V\subset X$, where $U$ is a Stein neighbourhood of the point $y$ in $Y$, and $V$ is a Stein neighbourhood of the set $K$ in $F_y$. It follows that for $q>0$: $$H^{q}(K,\mathcal{F}\mid_{Z_{y}})=\varinjlim\limits_{U\ni y, V\supset K}H^{q}(U\times V,\mathcal{F})=0.$$ 

Then the group cohomology $H^{\bullet}_{c}(Z_y,\mathcal{F}\mid_{Z_y})$ is obtained as cohomology of the \v Cech complex $C_{c}^{\bullet}(\mathfrak{K},\mathcal{F}\mid_{Z_{y}})$ where

$$C_{c}^{p}(\mathfrak{K},\mathcal{F}\mid_{Z_y}):=\bigoplus\limits_{s=(s_{0},\cdots,s_p)} \Gamma(K_{s_0}\cap\cdots\cap K_{s_p},\mathcal{F}\mid_{Z_y}).$$  

Note that $\Gamma(K_{s_0}\cap\cdots\cap K_{s_p},\mathcal{F}\mid_{Z_y})$ has a natural LF structure which is separated and DFS (see Remark \ref{propDFSstein}). So, $C_{c}^{p}(\mathfrak{K},\mathcal{F}\mid_{Z_y})$ is also separated and DFS, the group cohomology $H^{\bullet}_{c}(F_y,\mathcal{F}\mid_{Z_y})$ has a QDFS structure. 

\begin{remark}
Obviously, if $\phi\colon X\to Y$ is a holomorphic fiber bundle with fiber $F$ and $Z\subset F$ a complex submanifold, we obtain likewise that $H^{\bullet}_{c}(Z_y,\mathcal{F}\mid_{Z_y})$ has the natural QDFS-topology. 
\end{remark}

\subsection{Inductive topology on $H^{\bullet}_{c}(F_y,\mathcal{F}\mid_{F_y})$ with respect to the pair of spaces}
Using the long exact sequence of the pair we can define an inductive topology on $H^{\bullet}_{c}(F_y,\mathcal{F}\mid_{F_y})$. 

Let $X'=Y\times F'$ be the product of complex manifolds with countable basis, $\mathcal{F}$ be a coherent analytic sheaf on $X'$. Assume that $F'$ is a compact complex manifold and let $Z\subset F'$ be a proper closed complex submanifold. We put $F=F'\setminus Z$, $F_{y}:=\{y\}\times F$, $F'_{y}:=\{y\}\times F'$, $Z_{y}=F_{y}'\setminus F_{y}$ and $X=Y\times F$. 

For the pair $(F_{y},F_{y}')$ and the sheaf $\mathcal{F}\mid_{F_{y}'}$ we have the following short exact sequence of sheaves (here we denote by $i\colon Z_{y}\hookrightarrow F'_{y}$ the canonical embedding):

$$ 0\to (\mathcal{F}\mid_{F_{y}})^{F_{y}'}\to \mathcal{F}\mid_{F'_{y}}\to i_{*}(\mathcal{F}\mid_{Z_{y}})\to 0$$ which induces the following long exact sequence (recall that $(\mathcal{F}\mid_{F_{y}})^{F_{y}'}$ is the extension by zero of $\mathcal{F}\mid_{F_{y}}$):
 
\begin{multline}\label{exactpairsequence}
\xymatrix@C=0.5cm{
  \cdots \ar[r] & H^{p}(F_{y}',\mathcal{F}\mid_{F_{y}'}) \ar[r]^{f} & H^{p}(Z_{y},\mathcal{F}\mid_{Z_{y}}) \ar[r]^{g} & H^{p+1}_{c}(F_{y},\mathcal{F}\mid_{F_{y}}) \ar[r] & \\ \ar[r] & H^{p+1}(F_{y}',\mathcal{F}\mid_{F_{y}'}) \ar[r] & \cdots }
\end{multline} 

The spaces $H^{\bullet}(F_{y}',\mathcal{F}\mid_{F_{y}'})$, $H^{\bullet}(Z_{y},\mathcal{F}\mid_{Z_{y}})$ have canonical QDFS-topologies. Consider the cokernel of $f$ which is the quotient $$H^{p}(Z_{y},\mathcal{F}\mid_{Z_{y}})/f(H^{p}(F_{y}',\mathcal{F}\mid_{F_{y}'}))$$ with quotient topology. So, on the space $H^{p+1}_{c}(F_{y},\mathcal{F}\mid_{F_{y}})$ we may consider the inductive topology with respect to the canonical embedding $$H^{p}(Z_{y},\mathcal{F}\mid_{Z_{y}})/f(H^{p}(F_{y}',\mathcal{F}\mid_{F_{y}'}))\hookrightarrow H^{p+1}_{c}(F_{y},\mathcal{F}\mid_{F_{y}}).$$ 

This is the finest locally convex topology on $H^{p+1}_{c}(F_{y},\mathcal{F}\mid_{F_{y}})$ for which the canonical embedding is a continuous map (in particular, this topology is finer than the canonical QDFS-topology). This topology we call the inductive topology with repsect to the pair $(F'_{y},F_{y})$. 

Now, we may also define this topology using the notion of a mapping cone. Let $\mathfrak{K}'$ be a locally finite covering by Stein compacts of $F_{y}$ and $\mathfrak{K}''$ be a locally finite covering by Stein compacts of $Z_{y}$. Then $\mathfrak{K}:=\mathfrak{K}'\cup\mathfrak{K}''$ is a locally finite covering by Stein compacts of $F_{y}'$ (by Siu's theorem \cite[Theorem 3.1.1]{Forst}, each Stein compact of $Z_y$ is a Stein compact of $F_{y}'$). Note that for any $K'\in \mathfrak{K}'$ and for any $K''\in \mathfrak{K}''$ we have $K'\cap K''=\emptyset$.

We have the following short exact sequence of \v Cech complexes

$$0\to C_{c}^{\bullet}(\mathfrak{K}',\mathcal{F}\mid_{F_y})\overset{h}{\to} C_{c}^{\bullet}(\mathfrak{K},\mathcal{F}\mid_{F_y'})\overset{f}{\to} C_{c}^{\bullet}(\mathfrak{K}'',\mathcal{F}\mid_{Z_y})\to 0$$
which is split and also induces the long exact sequence (\ref{exactpairsequence}). 

Let $C^{\bullet}(h)$ be a cone of the map $h$. Recall that $$C^{p}(h)=C_{c}^{p+1}(\mathfrak{K}',\mathcal{F}\mid_{F_y})\oplus C_{c}^{p}(\mathfrak{K},\mathcal{F}\mid_{F_y'}), $$ $$\delta_{C(h)}=\begin{pmatrix}
-\delta'& 0\\
h & \delta
\end{pmatrix}$$
where $\delta', \delta$ are \v Cech differentials in $C_{c}^{\bullet}(\mathfrak{K}',\mathcal{F}\mid_{F_y}), C_{c}^{\bullet}(\mathfrak{K},\mathcal{F}\mid_{F_y'})$ respectively.

There exist the following maps: $$\alpha\colon C_{c}^{\bullet}(\mathfrak{K},\mathcal{F}\mid_{F_y'})\to C^{\bullet}(h), y\to (0,y),$$ $$\beta\colon C^{\bullet}(h)\to C_{c}^{\bullet+1}(\mathfrak{K}',\mathcal{F}\mid_{F_y}), (x,y)\to -x.$$ 

We have an algebraical quasi-isomorphism of complexes: $$s\colon C^{\bullet}(h)\to C_{c}^{\bullet}(\mathfrak{K}'',\mathcal{F}\mid_{Z_y}), (x,y)\to f(y).$$ 

We obtain the following isomorphism of exact sequences: 
\[\begin{diagram}\tiny
\node{\cdots} \arrow{e,t}{} \node{H^{p}(F_{y}',\mathcal{F}\mid_{F_{y}'})} \arrow{e,t}{f} \arrow{se,b}{\alpha} \node{H^{p}(Z_{y},\mathcal{F}\mid_{Z_{y}})} \arrow{e,t}{g} \node{H^{p+1}_{c}(F_{y},\mathcal{F}\mid_{F_{y}})}\arrow{e,t}{}\arrow{se,b}{} \node{\cdots}
\\
\node{\cdots} \arrow{ne,r}{} \node{} \node{H^{p}(C^{\bullet}(h))} \arrow{n,r}{s=iso} \arrow{ne,r}{\beta} \node{}\node{\cdots}
\end{diagram}\]
Note that the $s$ induces a topological isomorphism to cohomology. It follows that the inductive topology on $H^{p+1}_{c}(F_{y},\mathcal{F}\mid_{F_{y}})$ with respect to the pair $(F_{y}',F_{y})$ is the inductive topology with respect to the canonical embedding $$\beta \colon H^{p}(C^{\bullet}(h))/ \alpha(H^{p}(F_{y}',\mathcal{F}\mid_{F_{y}'}))\hookrightarrow H^{p+1}_{c}(F_{y},\mathcal{F}_{F_{y}}).$$

\subsection{The duality theorems}

For the complex manifolds there are the duality theorems due to Serre and Malgrange. We need the result for the structure sheaf.

\begin{theorem} \cite[Chapter VII, Theorems 4.1, 4.2]{BanStan}\label{duality}
Let $X$ be a complex manifold with countable basis.
\begin{enumerate}
\item The associated separated spaces of $H^{p}_{c}(X,\mathcal{O}_{X})$ and $H^{n-p}(X,\Omega^{n}_{X})$ are in  topological duality;
\item The space $H^{p}_{c}(X,\mathcal{O}_{X})$ is separated if and only if the space $H^{n-p+1}(X,\Omega^{n}_{X})$ is separated. 
\end{enumerate}
\end{theorem}

\subsection{The nuclear spaces and topological K\"unnet formula}

Now we recall the topological tensor product and the K\"unnet formula. For more details about the topological tensor product see \cite{Dem, Pietsch, Schaefer}, about the K\"unnet theorem for topological tensor product see \cite{Kaup} or \cite[Chapter IX]{Dem}. 

Let $E, F$ be locally convex topological vector space. The topological tensor product $E\widehat{\otimes}_{\pi}F$ (resp. $E\widehat{\otimes}_{\varepsilon}F$) is the Hausdorff completion $E\otimes F$, equipped with the family of semi-norms $p\widehat{\otimes}_{\pi}q$ (resp. $p\widehat{\otimes}_{\varepsilon}q$) associated to fundamental families of semi-norms on $E$ and $F$ which is defined by formulas: 

$$p\widehat{\otimes}_{\pi}q(t)=\inf\big\{\sum\limits_{1\leq j\leq N}p(x_{j})q(y_{j})\mid t=\sum\limits_{1\leq j\leq N}x_{j}\otimes y_{j}, x_{j}\in E, y_{j}\in F\big\},$$
$$p\widehat{\otimes}_{\varepsilon}q(t)=\sup\big\{\mid l\otimes m(t)\mid\mid l\in E', m\in F', \parallel l\parallel_{p}\leq 1, \parallel m\parallel_{q}\leq 1\big\}.$$

A morphism $f\colon E\to F$ of complete locally convex spaces is said to be nuclear if $f$ can be written as $f(x)=\sum \lambda_{j}\varepsilon_{j}(x)y_j$, where $\{\lambda_j\}$ is a sequence of scalar with $\sum \mid\lambda_j\mid<\infty$, $\varepsilon_j\in E'$ an equicontinuous sequence
of linear forms and $y_{j}\in F$ a bounded sequence. 

A locally convex space $E$ is called nuclear if every continuous linear map of $E$ into any Banach space is nuclear.

If $E$ is an FS space and $\{p_j\}$ is an increasing sequence of semi-norms on $E$ defining the topology of $E$, we have $E=\varprojlim \widehat{E}_{p_j}$, where $\widehat{E}_{p_j}$ is the Hausdorff completion of $(E,p_{j})$ and $\widehat{E}_{p_{j+1}}\to \widehat{E}_{p_j}$ the canonical morphism. 

An FS space $E$ is nuclear if and only if the topology of $E$ can be defined by an increasing sequence of semi-norms $p_j$ such that each canonical morphism $\widehat{E}_{p_{j+1}}\to \widehat{E}_{p_j}$ of Banach spaces is nuclear. 

Now we list some facts about nuclear spaces.

\begin{itemize}
\item If $E$ or $F$ is nuclear, we obtain an isomorphism $E\widehat{\otimes}_{\pi}F\cong E\widehat{\otimes}_{\varepsilon}F$;
\item If $E,F$ are nuclear spaces, then $E\widehat{\otimes}F$ is nuclear;
\item Every subspace and every separated quotient space of a nuclear space is nuclear;
\item The product of an arbitrary family of nuclear spaces is nuclear;
\item The locally convex direct sum of a countable family of nuclear spaces is a nuclear space;
\item The projective limit of any family of nuclear spaces, and the inductive limit of a countable family of nuclear spaces, are nuclear;
\item Let $0\to E_1\to E_2\to E_3\to 0$ be an exact sequence of topological homomorphisms between Fr\'echet spaces and let $F$ be a Fr\'echet space. If $E_2$ or $F$ is nuclear, there is an exact sequence of topological homomorphisms $$0\to E_1\widehat{\otimes}F\to E_2\widehat{\otimes}F\to E_3\widehat{\otimes}F\to 0;$$
\end{itemize}

\begin{example}\label{remntensor}
\begin{enumerate}
\item For discrete spaces $I,J$ there is an isometry $$l^{1}(I)\widehat{\otimes}_{\pi}l^{1}(J)\cong l^{1}(I\times J);$$
\item If $X, Y$ are compact topological spaces and $C(X), C(Y)$ are their algebras of continuous functions with sup-norms, then $C(X)\widehat{\otimes}_{\varepsilon}C(Y)\cong C(X\times Y)$;
\item Let $D$ be a polydisk in $\mathbb{C}^{n}$. Then $\mathcal{O}(D)$ is a nuclear FS space. 
\item If $\mathcal{F}$ is a coherent analytic sheaf on a complex analytic variety $X$, then $\mathcal{F}(X)$ is a nuclear space. 
\item Let $\mathcal{F}, \mathcal{G}$ be coherent analytic sheaves on complex analytic varieties $X,Y$ respectively. Then there is a canonical topological isomorphism $$\mathcal{F}\boxtimes\mathcal{G}(X\times Y)\cong \mathcal{F}(X)\widehat{\otimes}\mathcal{G}(Y)$$ where $\mathcal{F}\boxtimes\mathcal{G}=p_{1}^{*}\mathcal{F}\otimes_{\mathcal{O}_{X\times Y}}p_{2}^{*}\mathcal{G}$ is the analytic external tensor product of $\mathcal{F}$ and $\mathcal{G}$. In particular, $$\mathcal{O}_{X\times Y}(X\times Y)\cong \mathcal{O}_{X}(X)\widehat{\otimes}\mathcal{O}_{Y}(Y).$$
\item Let $\mathcal{F}$ be a coherent analytic sheaf on a complex analytic variety $X$. If the spaces $H^{\bullet}(X,\mathcal{F})$, $H^{\bullet}_{c}(X,\mathcal{F})$ are separated, then they are nuclear spaces. 
\item Let $X=Y\times F$ be the product of complex manifolds, $\mathcal{F}$ be a coherent analytic sheaf on $X$, $Z\subset F$ be a closed complex submanifold, $Z_{y}=Z\times \{y\}$. If the space $H^{\bullet}_{c}(Z_y,\mathcal{F}\mid_{Z_y})$ is separated, then it is a nuclear space. 
\end{enumerate}
\end{example}

We need the following proposition.

\begin{proposition}\cite[Chapter IX, Theorem 5.23]{Dem}, \cite[Theorem I]{Kaup}\label{Kunnet}
Let $\mathcal{F},\mathcal{G}$ be coherent analytic sheaves over complex analytic varieties and suppose that the cohomology spaces $H^{\bullet}(X,\mathcal{F})$ and $H^{\bullet}(Y,\mathcal{G})$ are separated spaces. Then there is a topological isomorphism $$\bigoplus\limits_{p+q=n}H^{p}(X,\mathcal{F})\widehat{\otimes}H^{q}(Y,\mathcal{G})\cong H^{n}(X\times Y,\mathcal{F}\boxtimes\mathcal{G}).$$
\end{proposition}

\subsection{On separatedness of the space $H^{\bullet}_{c}(F_{y},\mathcal{O}_{X}\mid_{F_{y}})$}

Now let $X=Y\times F$ be the product of complex manifolds with countable basis, $Y$ be a Stein manifold. Let $\Omega^{p}$ be the sheaf of holomorphic $p$-form. 

\begin{lemma}
There exists an isomorphism $\Omega_{X}^{n}\cong\bigoplus\limits_{p+q=n}\Omega^{p}_{Y}\boxtimes\Omega^{q}_{F}$ as Fr\'echet-Schwartz sheaves. 
\end{lemma}
\begin{proof}
Let $p_{1}\colon X\to Y$ and $p_{2}\colon X\to F$ be canonical projections. We have the following canonical  algebraic isomorphisms:
\begin{multline*}
\Omega^{n}_{X}=\wedge^{n}\Omega^{1}_{X}=\wedge^{n}(p_{1}^{*}\Omega^{1}_{Y}\oplus p_{2}^{*}\Omega^{1}_{F})\cong \\\cong \bigoplus\limits_{p+q=n}\wedge^{p}(p_{1}^{*}\Omega^{1}_{Y})\otimes \wedge^{q} (p_{2}^{*}\Omega^{1}_{F})\cong \bigoplus\limits_{p+q=n}p_{1}^{*}\Omega^{p}_{Y}\otimes p_{2}^{*}\Omega^{q}_{F}=\bigoplus\limits_{p+q=n}\Omega^{p}_{Y}\boxtimes\Omega^{q}_{F}.
\end{multline*} 

Let $d_{X, n}=\dim (\bigwedge^{n}(\mathbb{C}^{\dim X}))$ , $d_{Y,p}=\dim (\bigwedge^{p}(\mathbb{C}^{\dim Y}))$, $d_{F,q}=\dim (\bigwedge^{q}(\mathbb{C}^{\dim F}))$. Note that $d_{X,n}=\sum\limits_{p+q=n}d_{Y,p}d_{F,q}$.

Locally, for an open set $U\times V\subset X$ we have the topological isomorphisms: 
\begin{multline*}
\Omega_{X}^{n}(U\times V)\cong(\mathcal{O}_{X}(U\times V))^{\oplus d_{X, n}}\cong (\mathcal{O}_{Y}(U)\widehat{\otimes}\mathcal{O}_{F}(V))^{\oplus d_{X, n}} \cong \\\cong \bigoplus\limits_{p+q=n} (\mathcal{O}_{Y}(U))^{\oplus d_{Y, p}}\widehat{\otimes}(\mathcal{O}_{F}(V))^{\oplus d_{F, q}}\cong
\bigoplus\limits_{p+q=n}\Omega^{p}_{Y}(U)\widehat{\otimes}\Omega^{q}_{F}(V).
\end{multline*}

It follows that for any open set $W\subset X$ we have topological isomorphism $$\Omega_{X}^{n}(W)\cong
\bigoplus\limits_{p+q=n}(\Omega^{p}_{Y}\boxtimes\Omega^{q}_{F})(W).$$ It concludes the proof of Lemma.

\end{proof}

\begin{proposition}\label{lemmaseparated}
Let $y\in Y$ be any point. If for every $k,l$ the spaces $H^{k}(F,\Omega^{l}_{F})$ are separated, then 
\begin{enumerate}
\item for any $i>0$ the spaces $H^{i}_{c}(Y\setminus\{y\}\times F,\mathcal{O}_{X})$, $H^{i}_{c}(Y\times F,\mathcal{O}_{X})$ are separated;
\item for any $i>0$ the canonical map $H^{i}_{c}(Y\setminus\{y\}\times F,\mathcal{O}_{X})\to H^{i}_{c}(Y\times F,\mathcal{O}_{X})$ is surjective. 
\end{enumerate}
\end{proposition}

\begin{proof}
Since $Y$ is a Stein space, it follows that $H^{i}(Y,\Omega^{p}_{Y})=0$ for all $i>0$ and for all $p$, $\Omega^{p} (Y)$ is separated for all $p$ and $H^{j}(Y\setminus\{y\},\Omega^{q}_{Y})$ is separated for all $j\geq 0$ and for all $q$. 

By the K\"unnet formula (Proposition \ref{Kunnet}) we obtain the following topological isomorphisms:
 
 $$\Omega^{p}(Y)\widehat{\otimes}H^{n-i}(F,\Omega^{q}_{F})\cong H^{n-i}(Y\times F,\Omega^{p}_{Y}\boxtimes\Omega^{q}_{F});$$
$$\bigoplus\limits_{k+l=n-i}H^{k}(Y\setminus\{y\},\Omega^{p}_{Y})\widehat{\otimes}H^{l}(F,\Omega^{q}_{F})\cong H^{n-i}(Y\setminus\{y\}\times F,\Omega^{p}_{Y}\boxtimes\Omega^{q}_{F}).$$

It follows that for any $i\geq 0$ the spaces $$H^{n-i}(Y\setminus\{y\}\times F,\Omega^{n}_{X}), H^{n-i}(Y\times F,\Omega^{n}_{X})$$ are separated. By the duality theorem \ref{duality} we have that the spaces $$H^{i}_{c}(Y\setminus\{y\}\times F,\mathcal{O}_{X}), H^{i}_{c}(Y\times F,\mathcal{O}_{X})$$ are separated. Moreover, we obtain the topological isomorphisms: 

$$H^{i}_{c}(Y\setminus\{y\}\times F,\mathcal{O}_{X})\cong (H^{n-i}(Y\setminus\{y\}\times F,\Omega^{n}_{X}))',$$

$$H^{i}_{c}(Y\times F,\mathcal{O}_{X})\cong (H^{n-i}(Y\times F,\Omega^{n}_{X}))',$$

Since $\Omega^{p}(Y)$ is a closed subspace of a nuclear space $\Omega^{p}(Y\setminus \{y\})$, it follows that we have the injective topological homomorphism $$0\to \Omega^{p}(Y)\widehat{\otimes}H^{n-i}(F,\Omega^{q}_{F})\to \Omega^{p}(Y\setminus \{y\})\widehat{\otimes}H^{n-i}(F,\Omega^{q}_{F}).$$

So, the canonical map $$H^{n-i}(Y\times F,\Omega^{n}_{X})\to H^{n-i}(Y\setminus\{y\}\times F,\Omega^{n}_{X})$$ is injective topological homomorphism for all $i$. It follows that the dual map is surjective. This means that for any $i>0$ the canonical map $$H^{i}_{c}(Y\setminus\{y\}\times F,\mathcal{O}_{X})\to H^{i}_{c}(Y\times F,\mathcal{O}_{X})$$ is surjective.
\end{proof}

In particular we have the following

\begin{corollary}\label{corolseparated}
Let $F_{y}=\{y\}\times F$. If for every $k,l$ the spaces $H^{k}(F,\Omega^{l}_{F})$ are separated, then $H^{i}_{c}(F_{y},\mathcal{O}_{X}\mid_{F_{y}})$ are separated spaces for all $i$. 
\end{corollary}
\begin{proof}

Since $Y$ is a noncompact manifold, it follows that $$H_{c}^{0}(Y\setminus\{y\}\times F,\mathcal{O}_{X})=H^{0}_{c}(Y\times F,\mathcal{O}_{X})=0.$$ We have the following long exact sequence of the pair $(X,X\setminus F_{y})$ \cite[Chapter II, Section 10.3]{Bredon}:

\begin{multline}\label{pairexact}
\xymatrix@C=0.5cm{
  0  \ar[r] & H^{0}_{c}(F_{y},\mathcal{O}_{X}\mid_{F_{y}}) \ar[r] & H^{1}_{c}(Y\setminus\{y\}\times F,\mathcal{O}_{X}) \ar[r]^{f_1} & H^{1}_{c}(Y\times F,\mathcal{O}_{X}) \ar[r]  & \\ \ar[r]  & H^{1}_{c}(F_{y},\mathcal{O}_{X}\mid_{F_{y}}) \ar[r] & H^{2}_{c}(Y\setminus\{y\}\times F,\mathcal{O}_{X}) \ar[r]^{f_2}& H^{2}_{c}(Y\times F,\mathcal{O}_{X}) \ar[r] & \cdots   }
\end{multline} 

Since by Proposition \ref{lemmaseparated} we have that $f_i$ are surjective continuous maps between separated spaces, so $\ker f_i$ are closed subspaces. It follows that $H^{i}_{c}(F_{y},\mathcal{O}_{X}\mid_{F_{y}})$ are separated spaces. 

\end{proof}

\section{Density lemma for holomorphic fiber bundle and Stein fibers case}

Let $Y,F$ be complex manifolds with countable basis. Let $X=Y\times F$, $F_{y}=\{y\}\times F$, let $Z_{y}\subset F_{y}$ be a closed complex submanifold. Let $p_{1}\colon X\to Y$, $p_{2}\colon X\to F$ be canonical projections.

We have the canonical homomorphisms of sheaves $$p_{1}^{-1}\mathcal{O}_{Y}\hookrightarrow \mathcal{O}_{X}, p_{2}^{-1}\mathcal{O}_{F}\hookrightarrow\mathcal{O}_{X},$$ and we obtain the canonical injective homomorphism of sheaves $$p_{1}^{-1}\mathcal{O}_{Y}\otimes p_{2}^{-1}\mathcal{O}_{F}\hookrightarrow\mathcal{O}_{X}, f\otimes g\to fg.$$ 

For any open set of the form $U\times V\in X$ we obtain that the injective map $$\mathcal{O}_{Y}(U)\otimes \mathcal{O}_{F}(V)\hookrightarrow\mathcal{O}_{X}(U\times V)$$ has a dense image (Remark \ref{remntensor}). 

Note that we have the canonical isomorphisms of the sheaves $$(p_{1}^{-1}\mathcal{O}_{Y}\otimes p_{2}^{-1}\mathcal{O}_{F})\mid_{Z_y}\cong \mathcal{O}_{Y,y}\otimes \mathcal{O}_{F_{y}}\mid_{Z_y}$$ where $\mathcal{O}_{Y,y}$ considered as the constant sheaf on the fiber $F_y$. 

For any open subset $W\subset Z_{y}$ we obtain that the canonical injective map 
$$\mathcal{O}_{Y,y}\otimes \mathcal{O}_{F_y}\mid_{Z_y}(W)\hookrightarrow \mathcal{O}_{X}\mid_{Z_y}(W)$$ has a dense image. 

\begin{lemma}\label{Lemmaondense}
The canonical homomorphism $$\mathcal{O}_{Y, y}\otimes H^{\bullet}_{c}(Z_{y},\mathcal{O}_{F_y}\mid_{Z_y})\to H^{\bullet}_{c}(Z_y,\mathcal{O}_{X}\mid_{Z_y})$$ has a dense image with respect to the QDFS-topology on $H^{\bullet}_{c}(Z_y,\mathcal{O}_{X}\mid_{Z_y})$. 
\end{lemma}

\begin{proof}

Let $\mathfrak{K}$ be a locally finite covering of $Z_y$ by Stein compacts. Then the group cohomology $H^{\bullet}_{c}(Z_y,\mathcal{O}_{X}\mid_{Z_y})$ is obtained as cohomology of the \v Cech complex $C_{c}^{\bullet}(\mathfrak{K},\mathcal{O}_{X}\mid_{Z_{y}})$ where

$$C_{c}^{p}(\mathfrak{K},\mathcal{O}_{X}\mid_{Z_y}):=\bigoplus\limits_{s=(s_{0},\cdots,s_p)} \Gamma(K_{s_0}\cap\cdots\cap K_{s_p},\mathcal{O}_{X}\mid_{Z_y}).$$  

Now, we consider the sheaf $\mathcal{O}_{Y,y}\otimes \mathcal{O}_{F_{y}}\mid_{Z_y}$ on $Z_y$. 

By the universal coefficient theorem \cite[Theorem 15.3]{Bredon} and by Siu's theorem (see \cite[Theorem 3.1.1]{Forst}) for any Stein compact $K\in\mathfrak{K}$ we have the following:
\begin{itemize}
\item for global sections 
$$\Gamma(K,\mathcal{O}_{Y,y}\otimes \mathcal{O}_{F_{y}}\mid_{Z_{y}})\cong \mathcal{O}_{Y,y}\otimes \Gamma(K,\mathcal{O}_{F_{y}}\mid_{Z_y});$$
\item for group cohomologies $$H^{q}(K,\mathcal{O}_{Y,y}\otimes \mathcal{O}_{F_{y}}\mid_{Z_{y}})\cong \mathcal{O}_{Y,y}\otimes H^{q}(K,\mathcal{O}_{F_{y}}\mid_{Z_{y}})=0.$$
\end{itemize}

For the sheaf $\mathcal{O}_{Y,y}\otimes \mathcal{O}_{F_{y}}\mid_{Z_y}$ on $Z_y$ we have the \v Cech complex $$C_{c}^{\bullet}(\mathfrak{K},\mathcal{O}_{Y,y}\otimes \mathcal{O}_{F_{y}}\mid_{Z_y})=\mathcal{O}_{Y,y}\otimes C_{c}^{\bullet}(\mathfrak{K}, \mathcal{O}_{F_{y}}\mid_{Z_y})$$ 

The group $H^{\bullet}_{c}(F_y,\mathcal{O}_{Y,y}\otimes \mathcal{O}_{F_{y}}\mid_{Z_y})$ is obtained as cohomology of the \v Cech complex $\mathcal{O}_{Y,y}\otimes C_{c}^{\bullet}(\mathfrak{K}, \mathcal{O}_{F_{y}}\mid_{Z_y})$. Moreover, we have the canonical isomorphism 
$$H^{\bullet}_{c}(F_{y},\mathcal{O}_{Y,y}\otimes \mathcal{O}_{F_{y}}\mid_{Z_{y}})=\mathcal{O}_{Y,y}\otimes H^{\bullet}_{c}(F_{y},\mathcal{O}_{F_{y}}\mid_{Z_{y}}).$$

Note that the canonical injective morphism of complexes $$ \mathcal{O}_{Y,y}\otimes C_{c}^{\bullet}(\mathfrak{K}, \mathcal{O}_{F_{y}}\mid_{Z_y})\hookrightarrow C_{c}^{\bullet}(\mathfrak{K},\mathcal{O}_{X}\mid_{Z_{y}})$$ has a dense image. 

Now, every cochain $f=(f_{s_0,\cdots,s_{p}})$ in $C_{c}^{p}(\mathfrak{K},\mathcal{O}_{X}\mid_{Z_{y}})$ can be represented as $f=\sum\limits_{I\in \mathbb{Z}_{\geq 0}^{m}}f_{I}z^{I}$, where $z=(z_{1},\cdots,z_m)$ is an local coordinates with origin at the point $y\in Y$, $f_{I}=(f_{I, (s_0,\cdots,s_{p})})\in C_{c}^{p}(\mathfrak{K}, \mathcal{O}_{F_{y}}\mid_{Z_{y}})$ and the power series $\sum\limits_{I\in \mathbb{Z}_{\geq 0}^{m}}f_{I, (s_0,\cdots, s_p)}z^{I}$ uniformly converges on compact sets in $U\times V$ (here $U\subset Y$ is a small polydisk with center $y\in Y$, $V$ is a Stein neigbourhood of $K_{s_0,\cdots s_p}$ in $F_{y}$).

Assume that $\delta(f)=0$; this means that for any $(s_{0},\cdots, s_{p+1})$ we have 
\begin{multline}
0=\delta(f)_{s_{0},\cdots, s_{p+1}}=\sum\limits_{k=0}^{p+1}(-1)^{k} \big(\sum\limits_{I\in \mathbb{Z}_{\geq 0}^{m}}f_{I,(s_{0},\cdots,\widehat{k},\cdots,s_{p+1})}z^{I}\big)=\\=\sum\limits_{I\in \mathbb{Z}_{\geq 0}^{m}} \big(\sum\limits_{k=0}^{p+1}(-1)^{k}f_{I, (s_{0},\cdots,\widehat{k},\cdots,s_{p+1})}\big)z^{I}.
\end{multline}

Therefore $\delta(f_{I})=0$. It follows that for every $q$ the canonical injective morphism $$ \mathcal{O}_{Y,y}\otimes Z_{c}^{q}(\mathfrak{K}, \mathcal{O}_{F_{y}}\mid_{Z_y})\hookrightarrow Z_{c}^{q}(\mathfrak{K},\mathcal{O}_{X}\mid_{Z_{y}})$$ has a dense image. 

Therefore the canonical homomorphism $$\mathcal{O}_{Y,y}\otimes H^{\bullet}_{c}(F_{y},\mathcal{O}_{F_{y}}\mid_{Z_y})\to H^{\bullet}_{c}(F_y,\mathcal{O}_{X}\mid_{Z_y})$$ has a dense image. 
\end{proof}

Let $X,Y,F$ be complex manifolds with countable basis and $\phi\colon X\to Y$ be a holomorphic fiber bundle with connected fiber $F$. Recall that by Remark \ref{directimage} we have the canonical isomorphism $(R^{\bullet}\phi_{!}\mathcal{O}_{X})_{y}\cong H^{\bullet}_{c}(F_y,\mathcal{O}_{X}\mid_{F_y})$, where $F_{y}=\phi^{-1}(y)$.

\begin{corollary}
Let $\phi\colon X\to Y$ be a holomorphic fiber bundle with fiber $F$. Let $F_{y}=\phi^{-1}(y)$. Then for any $y\in Y$ there is the canonical homomorphism $$\mathcal{O}_{Y, y}\otimes H^{\bullet}_{c}(F_{y},\mathcal{O}_{F_y})\to H^{\bullet}_{c}(F_y,\mathcal{O}_{X}\mid_{F_y})$$ which has a dense image with respect to the QDFS-topology on $H^{\bullet}_{c}(F_y,\mathcal{O}_{X}\mid_{F_y})$. 
\end{corollary}
\begin{proof}
The problem is local. We may assume that $X=Y\times F$ and apply Lemma \ref{Lemmaondense} for the case $Z=F$. 
\end{proof}

\begin{corollary}
Let $\phi\colon X\to Y$ be a holomorphic fiber bundle with noncompact fiber $F$. Then $\phi_{!}\mathcal{O}_{X}=0$. In particular, $H^{1}_{c}(X,\mathcal{O}_{X})\cong \Gamma_{c}(Y,R^{1}\phi_{!}\mathcal{O}_{X})$. 
\end{corollary}
\begin{proof}
Let $F_{y}=\phi^{-1}(y)$. We have $(\phi_{!}\mathcal{O}_{X})_{y}\cong \Gamma_{c}(F_{y}, \mathcal{O}_{X}\mid_{F_{y}})$. Since $\Gamma_{c}(F_{y}, \mathcal{O}_{X}\mid_{F_{y}})$ is separated DFS space and it has the dense subset $\mathcal{O}_{y}\otimes \Gamma_{c}(F_{y},\mathcal{O}_{F_y})=0$, it follows that $\phi_{!}\mathcal{O}_{X}=0$. By Corollary \ref{corcompactspectral} we obtain the last statement.
\end{proof}

So, we obtain the following vanishing result.
\begin{corollary}\label{vanishtheorem}
Let $\phi\colon X\to Y$ be a holomorphic fiber bundle with noncompact fiber $F$. Assume that for any $i<q$ sheaves $R^{i}\phi_{!}\mathcal{O}_{X}$ are separated sheaves. If $H^{i}_{c}(F,\mathcal{O}_{F})=0$ for any $i<q$, then $H^{i}_{c}(X,\mathcal{O}_{X})=0$ for any $i<q$. 
\end{corollary}

\begin{proof}
By Lemma \ref{Lemmaondense} we have $R^{i}\phi_{!}\mathcal{O}_{X}=0$. By Corollary \ref{corcompactspectral} we obtain $H^{i}(X,\mathcal{O}_{X})=0$ for any $i<q$. 
\end{proof}

If the fiber $F$ is a Stein manifold then the spaces $H^{k}(F,\Omega^{l}_{F})$ are separated. We obtain the following vanishing result.

\begin{corollary}\label{Stein}
Let $\phi\colon X\to Y$ be a holomorphic fiber bundle with fiber $F$ which is a Stein manifold. Then $H^{i}_{c}(X,\mathcal{O}_{X})=0$ for all $i<\dim F$.
\end{corollary}
\begin{proof}
We may assume that $\phi\colon X\to Y$ is a trivial holomorphic fiber bundle over a Stein manifold $Y$. This means that $X=Y\times F$ and $\phi$ is a projection onto the first factor $Y$. We put $F_{y}:=\{y\}\times F$.

By Corollary \ref{corolseparated} we have that $H^{i}_{c}(F_y,\mathcal{O}_{X}\mid_{F_y})$ are separated spaces for any $q$. Since $F$ is a Stein manifold, $H^{i}_{c}(F,\mathcal{O}_{F})=0$ for all $i<\dim F$. Then by Corollary \ref{vanishtheorem} we obtain that $H^{i}_{c}(X,\mathcal{O}_{X})=0$ for all $i<\dim F$.
\end{proof}

In particular, for $F=\mathbb{C}^{n}$ and $n>1$ we obtain Dwilewicz's result about vanishing $H_{c}^{1}(X,\mathcal{O}_{X})$ for the total space $X$ of a complex fiber bundle \cite[Corollary 1.4]{Dwilewicz}. 

\begin{remark}
Note that, the space $\mathcal{O}_{Y, y}$ is nuclear. Assume that the spaces $$H^{\bullet}_{c}(Z_y,\mathcal{O}_{X}\mid_{Z_y}), H^{\bullet}_{c}(Z_{y},\mathcal{O}_{F_y}\mid_{Z_y})$$ are separated spaces.

For any open subset $W\subset Z_{y}$ the canonical injective map 
$$\mathcal{O}_{Y,y}\otimes \mathcal{O}_{F_y}\mid_{Z_y}(W)\hookrightarrow \mathcal{O}_{X}\mid_{Z_y}(W)$$ is continuous with respect to the projective tensor product topology on $$\mathcal{O}_{Y,y}\otimes \mathcal{O}_{F_y}\mid_{Z_y}(W)$$ and the canonical DFS-topology on $\mathcal{O}_{X}\mid_{Z_y}(W)$.  

It follows that the canonical map $$\mathcal{O}_{Y, y}\otimes H^{\bullet}_{c}(Z_{y},\mathcal{O}_{F_y}\mid_{Z_y})\to H^{\bullet}_{c}(Z_y,\mathcal{O}_{X}\mid_{Z_y})$$ is continuous with respect to the projective tensor product topology and QDFS-topology. Then the canonical continuous map $$\mathcal{O}_{Y, y}\widehat{\otimes} H^{\bullet}_{c}(Z_{y},\mathcal{O}_{F_y}\mid_{Z_y})\to H^{\bullet}_{c}(Z_y,\mathcal{O}_{X}\mid_{Z_y})$$ is surjective.
\end{remark}

\section{Holomorphic fiber bundle with $(1,0)$-com\-pac\-ti\-fiable fibers}\label{sectioncompactif}

We introduce the following definition.

\begin{definition}\label{maindef}
A noncompact complex manifold $X$ is called $(b,\sigma)$-com\-pa\-cti\-fiable if it admits a compactification $X'$ with the following properties:
\begin{enumerate}
\item $X'$ is a compact complex manifold;
\item $X'\setminus X$ is a proper analytic subset and it has $b$ connected components;
\item $\dim_{\mathbb{C}} H^{1}(X',\mathcal{O}_{X'})=\sigma$.
\end{enumerate} 
\end{definition}

The numbers $(b,\sigma)$ for a complex manifold $X$ are independent on the compactification $X'$. More precisely, we have the following remark. 

\begin{remark}
\begin{enumerate}
\item If $X$ is a $(b,\sigma)$-com\-pa\-cti\-fiable complex manifold, then $X$ has exactly $b$ topological ends (about topological ends see \cite{Freudenthal}, for CW-complexes see \cite{Peschke});
\item The number $\sigma$ is a birational invariant (see, for example, \cite[Corollary 1.4]{Rao});
\item Let $A(X')$ be the Albanese manifold for a compact complex algebraic manifold $X'$, then $\dim A(X')=\sigma$. 
\end{enumerate}
\end{remark}

Let $X,Y,F$ be complex manifolds with countable basis. Now we have the following important lemma on the stalks of sheaf $R^{1}\phi_{!}\mathcal{O}_{X}$. Recall that $$(R^{1}\phi_{!}\mathcal{O}_{X})_{y}\cong H^{1}_{c}(F_y,\mathcal{O}_{X}\mid_{F_y}),$$ where $F_{y}=\phi^{-1}(y)$.

\begin{lemma}\label{thmseparate}
Let $\phi\colon X\to Y$ be a holomorphic fibre bundle with noncompact fiber $F$ and $F_{y}=\phi^{-1}(y)$. Assume that $F$ is $(1,0)$-compactifiable and $\dim F>1$. Then the inductive topology on $H^{1}_{c}(F_y,\mathcal{O}_{X}\mid_{F_y})$ with respect to the pair $(F_{y}',F_y)$ is separated and the canonical map $\mathcal{O}_{Y, y}\otimes H^{1}_{c}(F_{y},\mathcal{O}_{F_y})\to H^{1}_{c}(F_y,\mathcal{O}_{X}\mid_{F_y})$ is injective with dense image with respect to this topology. 
\end{lemma}
\begin{proof}
By assumption there exists a compact complex manifold $F'$ such that $Z=F'\setminus F$ is a connected proper analytic subset and $H^{1}(F',\mathcal{O}_{F'})=0$. From the long exact sequence for the pair $(F,F')$ \cite[Chapter II, Section 10.3]{Bredon} we obtain the short exact sequence:
$$\xymatrix@C=0.5cm{
  0  \ar[r] & \mathbb{C} \ar[r] & H^{0}(Z,\mathcal{O}_{F}\mid_{Z})  \ar[r] & H^{1}_{c}(F,\mathcal{O}_{F}) \ar[r] & 0.  }$$

Now, since the problem is local, we may assume that $\phi\colon X\to Y$ is a trivial holomorphic fiber bundle over a Stein manifold $Y$. This means that $X=Y\times F$ and $\phi$ is a projection onto the first factor $Y$. We have $F_{y}=\{y\}\times F$, $Z_{y}=F_{y}'\setminus F_{y}$ and $X'=Y\times F'$. 

For the pair $(F_{y},F_{y}')$ and the sheaf $\mathcal{O}_{X'}\mid_{F_{y}'}$ we have the following long exact sequence (note that $H^{0}_{c}(F_{y},\mathcal{O}_{X}\mid_{F_{y}})=0$, because $F_{y}$ is a noncompact complex manifold): 
$$\xymatrix@C=0.5cm{
  0  \ar[r] & H^{0}(F_{y}',\mathcal{O}_{X'}\mid_{F_{y}'}) \ar[r] & H^{0}(Z_{y},\mathcal{O}_{X'}\mid_{Z_{y}}) \ar[r] & H^{1}_{c}(F_{y},\mathcal{O}_{X}\mid_{F_{y}}) \ar[r] & \\ \ar[r] & H^{1}(F_{y}',\mathcal{O}_{X'}\mid_{F_{y}'}) \ar[r] & \cdots }$$

Recall that $$H^{i}(F_{y}',\mathcal{O}_{X'}\mid_{F_{y}'})\cong \varinjlim\limits_{U\supset y}H^{i}(U\times F',\mathcal{O}_{X'})$$ where the inductive limit may be taken over a cofinal system of Stein neighbourhoods of the point $y$. 

By the K\"unnet theorem for topological tensor product (Proposition \ref{Kunnet}) we obtain 
$$H^{0}(U\times F',\mathcal{O}_{X'})\cong H^{0}(U,\mathcal{O}_{Y})\widehat{\otimes} H^{0}(F',\mathcal{O}_{F'})=H^{0}(U,\mathcal{O}_{Y}),$$ 
$$H^{1}(U\times F',\mathcal{O}_{X'})\cong H^{1}(U,\mathcal{O}_{Y})\widehat{\otimes} H^{0}(F',\mathcal{O}_{F'})\oplus H^{0}(U,\mathcal{O}_{Y})\widehat{\otimes} H^{1}(F',\mathcal{O}_{F'})=0.$$ (by assumption, we have $H^{1}(F', \mathcal{O}_{F'})=0$, and for the Stein neighbourhood $U$ of the point $y$ we have $H^{1}(U,\mathcal{O}_{Y})=0$). Therefore, $$H^{0}(F_{y}',\mathcal{O}_{X'}\mid_{F_{y}'})\cong \varinjlim\limits_{U\supset y}H^{0}(U,\mathcal{O}_{Y})=\mathcal{O}_{Y,y},$$ $$H^{1}(F_{y}',\mathcal{O}_{X'}\mid_{F_{y}'})=0.$$

So, we obtain the short exact sequence 
$$\xymatrix@C=0.5cm{
  0  \ar[r] & \mathcal{O}_{Y,y} \ar[r] & H^{0}(Z_{y},\mathcal{O}_{X'}\mid_{Z_{y}})  \ar[r] & H^{1}_{c}(F_{y},\mathcal{O}_{X}\mid_{F_{y}}) \ar[r] & 0. }$$

For the algebraic tensor product $p_{1}^{-1}\mathcal{O}_{Y}\otimes p_{2}^{-1}\mathcal{O}_{F'}$ we have the following canonical isomorphisms  $$(p_{1}^{-1}\mathcal{O}_{Y}\otimes p_{2}^{-1}\mathcal{O}_{F'})\mid_{F_{y}'}\cong\mathcal{O}_{Y,y}\otimes \mathcal{O}_{F_{y}'},$$ $$(p_{1}^{-1}\mathcal{O}_{Y}\otimes p_{2}^{-1}\mathcal{O}_{F'})\mid_{Z_{y}}\cong\mathcal{O}_{Y,y}\otimes (\mathcal{O}_{F_{y}'}\mid_{Z_y}).$$

Apply the universal coefficient theorem \cite[Theorem 15.3]{Bredon} to the long exact sequence of the cohomology groups of the sheaf $\mathcal{O}_{Y,y}\otimes \mathcal{O}_{F_{y}'}$ under the pair $(F_{y},F_{y}')$  we also have the following short exact sequence. 
$$\xymatrix@C=0.5cm{
  0  \ar[r] & \mathcal{O}_{Y,y} \ar[r] & \mathcal{O}_{Y,y}\otimes H^{0}(Z_{y},\mathcal{O}_{F'_{y}}\mid_{Z_{y}})  \ar[r] & \mathcal{O}_{Y,y}\otimes H^{1}_{c}(F_{y},\mathcal{O}_{F_{y}}) \ar[r] & 0 }$$

Therefore we obtain the following commutative diagramm with exact rows and the maps $h_{1}, h_{2}$ have dense images (by Lemma \ref{Lemmaondense}). 

\[\tiny
\begin{diagram} 
\node{0\longrightarrow \mathcal{O}_{Y,y}} \arrow{e,t}{} \arrow{s,r}{\shortparallel}
\node{\mathcal{O}_{Y,y}\otimes H^{0}(Z_{y},\mathcal{O}_{F'_{y}}\mid_{Z_y})} \arrow{e,t}{} \arrow{s,r}{h_{1}}
\node{\mathcal{O}_{Y,y}\otimes H^{1}_{c}(F_{y},\mathcal{O}_{F_{y}})\longrightarrow 0} \arrow{s,r}{h_2} 
\\
\node{0\longrightarrow \mathcal{O}_{Y,y}\phantom{pp}} \arrow{e,t}{g_1} 
\node{\phantom{p}H^{0}(Z_{y},\mathcal{O}_{X'}\mid_{Z_{y}})\phantom{p}} \arrow{e,t}{} 
\node{\phantom{p}H^{1}_{c}(F_{y},\mathcal{O}_{X}\mid_{F_{y}})\longrightarrow 0}
\end{diagram}
\]

Note that $\mathcal{O}_{Y,y}$ is a closed subspace in $H^{0}(Z_{y},\mathcal{O}_{X'}\mid_{Z_{y}})$. Actually, let $\{f_{n}\}$ be a sequence of holomorphic function $f_n$ in a neigbourhood $U\times F$ of $F_{y}'$ which are constant on fibers. If $f=\lim\limits_{n\to\infty}f_n$ in a small neigbourhood $U'\times V$ of $Z_y$ (here $V\supset Z_{y}$, $y\in U'\subset U$), then $f$ is constant on fibers in $U'\times V$ over $U'$. It defines a function on $U'\times F$ which is constant on fibers. 

Note that since $h_{1}$ is injective, it follows that $h_{2}$ is also injective. 

The quotient space $H^{0}(Z_{y},\mathcal{O}_{X'}\mid_{Z_{y}})/\mathcal{O}_{Y,y}$ is a separated space which is algebraicaly isomorphic to $H^{1}_{c}(F_{y},\mathcal{O}_{X}\mid_{F_{y}})$. The inductive topology on $H^{1}_{c}(F_{y},\mathcal{O}_{X}\mid_{F_{y}})$ with respect to the pair $(F_{y}',F_{y})$ is exactly the separated quotient topology. Since $h_{1}$ has a dense image, it follows that the space $\mathcal{O}_{Y,y}\otimes H^{1}_{c}(F_{y},\mathcal{O}_{F_{y}})$ is dense in the separated quotient topology in $H^{1}_{c}(F_{y},\mathcal{O}_{X}\mid_{F_{y}})$. 

\end{proof}

\begin{corollary}\label{maincorcompactif}
Let $\phi\colon X\to Y$ be a holomorphic fibre bundle with $(1,0)$-com\-pac\-ti\-fiable fiber $F$, $\dim F>1$. If $H^{1}_{c}(F,\mathcal{O}_{F})=0$, then $H^{1}_{c}(X,\mathcal{O}_{X})=0$. 
\end{corollary}

\begin{proof}
By Lemma \ref{thmseparate} and Corollary \ref{corcompactspectral}. 
\end{proof}

\begin{remark}
Note that, in the case of $\dim F=1$ we have $H^{1}_{c}(F,\mathcal{O}_F)\neq 0$. It follows that $H^{1}_{c}(X,\mathcal{O}_{X})$ is not necessary trivial. For example, for the tautological line bundle $X=\mathcal{O}(-1)$ over projective space $\mathbb{P}^{n}$ we have $H^{1}_{c}(X,\mathcal{O}_{X})=0$, but for trivial line bundle $X=\mathbb{P}^{n}\times\mathbb{C}$ we have $$H^{1}_{c}(X,\mathcal{O}_{X})\cong H^{0}(\mathbb{P}^{n}\times\{\infty\},\mathcal{O}_{X}\mid_{\mathbb{P}^{n}\times\{\infty\}})/\mathbb{C}\cong \mathcal{O}_{\mathbb{P}(\mathbb{C}),\infty}/\mathbb{C}\neq 0.$$
\end{remark}

\section{Cohomological criterion for the Hartogs phenomenon}

J.-P. Serre proved the Hartogs phenomenon for Stein manifolds by using triviality of the cohomology group with compact supports $H^{1}_{c}(X,\mathcal{O}_X)$ where $\mathcal{O}_X$ is the sheaf of holomorphic functions \cite{Serre}.

We consider a class of normal complex analytic varieties such that for this class the finite dimensional of the group $H^{1}_{c}(X,\mathcal{O}_X)$ is a necessary and sufficient condition for the Hartogs phenomenon for $X$. 

First for a compact set $K$ of $X$ define $\mu(K)$ to be the union of $K$ with all connected components of $X\setminus K$ that are relatively compact in $X$ (see \cite{Viorel} and \cite[Chapter VII, Section D]{Ganning}). We call a complex analytic variety $X$ to be connected at boundary if for every compact $K$ of $X$ the set $X\setminus \mu(K)$ is connected. Define the space of germs of holomorphic functions on boundary $$\mathcal{O}_{X}(\partial X):=\varinjlim\limits_{S} H^{0}(X\setminus S, \mathcal{O}_{X}),$$ where the inductive limit is taken over all compact subsets $S$ of $X$ (or over a cofinal part of them) (see \cite[Chapter VII, Section D]{Ganning}).

We need the following lemmas. 
\begin{lemma}\label{Lemma1}
Let $X$ be a noncompact normal complex analytic variety which is connected at boundary. Let $K\subset X$ be a compact set such that $X\setminus K$ is connected. If $\dim H^{1}_{c}(X,\mathcal{O}_{X})<\infty$, then the restriction homomorphism $\mathcal{O}_{X}(X) \to\mathcal{O}_{X}(X\setminus K)$ is an isomorphism.
\end{lemma}

\begin{proof}

We consider the following exact sequence of cohomology groups \cite{BanStan}:

\begin{equation}\label{equat1}
	\xymatrix@C=0.5cm{
0 \ar[r] &  H^{0}_{K}(X,\mathcal{O}_{X}) \ar[r] &  H^{0}(X,\mathcal{O}_{X}) \ar[r]^{R_{K}} & H^{0}(X\setminus K,\mathcal{O}_{X}) \ar[r]^{F_{K}} & H^{1}_{K}(X,\mathcal{O}_{X}) \ar[r] & \cdots }
\end{equation}

Obviously, $H^{0}_{K}(X,\mathcal{O}_{X})=0$. So, the restriction homomorphism $R_{K}$ is injective. 

Recall that if $S,T\subset X$ are compact sets and $S\subset T$, then we obtain the canonical homomorphism $\phi_{ST}\colon H^{1}_{S}(X,\mathcal{O}_{X})\to H^{1}_{T}(X,\mathcal{O}_{X})$. Moreover, we have the canonical isomorphism \cite{BanStan} $$\varinjlim\limits_{S} H^{1}_{S}(X,\mathcal{O}_{X})\cong H^{1}_{c}(X,\mathcal{O}_{X})$$ where the inductive limit is taken over all compact subsets $S$ of $X$ (or over a cofinal part of them).

So, we obtain the following long exact sequence: 

\begin{equation}\label{seqexact2}
	\xymatrix@C=0.5cm{
0 \ar[r] &  \mathcal{O}_{X}(X) \ar[rr]^{r} && \mathcal{O}_{X}(\partial X) \ar[rr]^{c} && H^{1}_{c}(X,\mathcal{O}_{X}) \ar[r] & \cdots }
\end{equation}

Now we apply the Andreotti-Hill technique (see \cite[corollary 4.3]{AndrHill}). Let $m=\dim H^{1}_{c}(X,\mathcal{O}_{X})$. 

We can assume that $\mathcal{O}_{X}(\partial X)\neq\mathbb{C}$. Consider an equivalence class $f\in \mathcal{O}_{X}(\partial X)$ of a non-constant holomorphic function on $X\setminus K$. Suppose that $c(f)\neq 0$. We may assume that $c(f^{i})$ are non-zero for any $1\leq i\leq m+1$. The elements $$c(f), c(f^{2}),\cdots, c(f^{m+1})$$ are linearly dependent. This means that there exists a polynomial $P\in \mathbb{C}[T]$ of degree $m+1$ such that $c(P(f))=0$. It follows that there is a holomorphic function $H\in\mathcal{O}_{X}(X)$ which is non-constant and such that $r(H)=P(f)$. 

Now the elements $$ c(f),c(r(H)f),\cdots, c(r(H)^{m}f)$$ are also non-zero and there exists a polynomial $P_{1}\in \mathbb{C}[T]$ such that $c(P_{1}(r(H)f))=0$. It follows that there is a holomorphic function $F\in\mathcal{O}_{X}(X)$ such that $r(F)=P_{1}(r(H))f$. Setting $G=P_{1}(H)$, we obtain $r(F)=r(G)f$. 

Note that since $r(G^{m+1}H)=r(G^{m+1}P(F/G))$, it follows that $G^{m+1}H=G^{m+1}P(F/G)$. Hence on $X\setminus \{G=0\}$ we obtain $H=P(F/G)$. It follows that $$F/G\in \mathcal{O}_{X}(X\setminus \{G=0\})$$ and locally bounded in $X\setminus \{G=0\}$. Since $G\neq 0$, the Riemann extension theorem implies that $F/G\in \mathcal{O}_{X}(X)$. So, $r(F/G)=f$.

Therefore, the induced homomorphism $r\colon \mathcal{O}_{X}(X)\to \mathcal{O}_{X}(\partial X)$ is an isomorphism.

Let $f\in \mathcal{O}_{X}(X\setminus K)$. Let $t_{K}\colon \mathcal{O}_{X}(X\setminus K)\to \mathcal{O}_{X}(\partial X)$ be a natural homomorphism. There exists a function $\widehat{f}\in\mathcal{O}_{X}(X)$ such that $r(\widehat{f})=t_{K}(f)$.

\begin{equation}\label{diag1}
\begin{diagram}\tiny
\node[1]{\mathcal{O}_{X}(X)} \arrow[2]{e,t}{r_{K}} \arrow[1]{sse,r}{r} \arrow{se,r}{r_{K'}} 
\node[2]{\mathcal{O}_{X}(X\setminus K)} \arrow{sw,r}{r_{KK'}}\arrow{ssw,r}{t_{K}} \\
\node[2]{\mathcal{O}_{X}(X\setminus K')}\arrow[1]{s,r}{t_{K'}}\\
\node[2]{\mathcal{O}_{X}(\partial X)}
\end{diagram}
\end{equation}
 
Then there exists a compact set $K'\subset X$ such that $\widehat{f}\mid_{X\setminus K'}=f\mid_{X\setminus K'}$ (see diagramm \ref{diag1}). Recall that $X\setminus K$ is a connected set. Using the uniqueness theorem we obtain $\widehat{f}\mid_{X\setminus K}=f$.  The proof of the lemma is complete.
\end{proof}

In fact, applying the Andreotti-Hill technique to the exact sequence (\ref{equat1}) as in Lemma \ref{Lemma1} we obtain the following lemma. 

\begin{lemma}\label{Lemma2}
Let $X$ be a noncompact normal complex analytic variety. Let $K\subset X$ be a compact set such that $X\setminus K$ is connected. If $\dim H^{1}_{K}(X,\mathcal{O}_{X})<\infty$, then the restriction homomorphism $\mathcal{O}_{X}(X) \to\mathcal{O}_{X}(X\setminus K)$ is an isomorphism.
\end{lemma}

Now we have the following cohomological criterion for the Hartogs phenomenon.

 \begin{theorem}
 	\label{th1}
 	Let $X$ be a noncompact normal complex analytic variety which is connected at boundary. Assume that $X$ admits an open embedding (not necessary with dense image) $X\hookrightarrow X'$ into a topological space $X'$ and there exists a sheaf of $\mathbb{C}$-vector spaces $\mathcal{F}$ on $X'$ with $\dim H^{1}(X',\mathcal{F})<\infty$ and $\mathcal{F}\mid_{X}=\mathcal{O}_{X}$.	Then $X$ admits the Hartogs phenomenon if and only if $\dim H^{1}_{c}(X,\mathcal{O}_X)<\infty$.
\end{theorem}

\begin{proof}
Assume that $\dim H^{1}_{c}(X,\mathcal{O}_{X})<\infty$. Let $W\subset X$ be a domain and $K\subset W$ be a compact set such that $W\setminus K$ is connected. Note that since $X, W, W\setminus K$ are connected sets, it follows that $X\setminus K$ is connected set.

Further on, we have the following commutative diagram for $K\subset W\subset X\subset X'$ with exact rows. 

\[\tiny
\begin{diagram} 
\node[2]{\mathcal{F}(X')} \arrow{e,t}{r_1} \arrow{s,r}{}
\node{\mathcal{F}(X'\setminus K)} \arrow{e,t}{} \arrow{s,r}{}
\node{H^{1}_{K}(X',\mathcal{F})} \arrow{s,r}{h_1} \arrow{e,t}{}
\node{H^{1}(X',\mathcal{F})}
\\
\node{0}\arrow{e,t}{}
\node{\mathcal{O}_{X}(X)} \arrow{e,t}{r_{2}} \arrow{s,r}{}
\node{\mathcal{O}_{X}(X\setminus K)} \arrow{e,t}{} \arrow{s,r}{}
\node{H^{1}_{K}(X,\mathcal{O}_{X})} \arrow{s,r}{h_2} 
\\
\node{0}\arrow{e,t}{}
\node{\mathcal{O}_{X}(W)} \arrow{e,t}{r_3} 
\node{\mathcal{O}_{X}(W\setminus K)} \arrow{e,t}{} 
\node{H^{1}_{K}(W,\mathcal{O}_{X})} 
\end{diagram}
\]

Since $r_{2}$ is an isomorphism and $X\setminus K=(X'\setminus K)\cap X$, then $r_{1}$ is surjective. Therefore, $\dim H^{1}_{K}(X',\mathcal{O}_{X'})<\infty$. Using the excision property (see \cite{BanStan}) we obtain $h_1$ and $h_2$ are canonical isomorphisms. It follows that $\dim H^{1}_{K}(W,\mathcal{O}_{X})<\infty$. By Lemma \ref{Lemma2}, $r_3$ is an isomorphism. 

Now we assume that $X$ admits the Hartogs phenomenon. We define $$E^{0}(X'):=\varinjlim\limits_{K} H^{0}(X'\setminus K, \mathcal{F}),$$ and $$E^{1}_{c}(X'):=\varinjlim\limits_{K} H^{1}_{K}(X', \mathcal{F})$$ where the inductive limit is taken over all compact subsets $K$ of $X$ (note that not every compact set of $X'$ is a compact set in $X$).  The excision property (see \cite{BanStan}) implies that $$E^{1}_{c}(X')\cong H^{1}_{c}(X,\mathcal{O}_{X}).$$

Since for every compact set $K\subset X$ we have $X\setminus\mu(K)$ is a connected set, then the restriction homomorphism $\mathcal{O}_{X}(X)\to \mathcal{O}_{X}(X\setminus\mu(K))$ is an isomorphism. It follows that the restriction homomorphism $\mathcal{F}(X')\to \mathcal{F}(X'\setminus\mu(K))$ is a surjective map. 

We have the following exact sequence: 
\[
\begin{diagram} 
\node{\mathcal{F}(X')} \arrow{e,t}{r_{1}}
\node{E^{0}(X')} \arrow{e,t}{}
\node{E^{1}_{c}(X')}\arrow{e,t}{}
\node{H^{1}(X',\mathcal{F})}
\end{diagram}
\] where $r_1$ is surjective. It follows that $\dim H^{1}_{c}(X,\mathcal{O}_{X})\leq \dim H^{1}(X',\mathcal{F})<\infty$.  The proof of the theorem is complete.
\end{proof}

In particular, if $X$ as above in Theorem \ref{th1} and admits the Hartogs phenomenon, then $\dim H^{1}_{c}(X,\mathcal{O}_{X})\leq \dim H^{1}(X',\mathcal{F})$ and we obtain the following corollary.

\begin{corollary}\label{corth1}
Let $X$ be a noncompact normal complex analytic variety which is connected at boundary. Assume that $X$ admits an open embedding (not necessary with dense image) $X\hookrightarrow X'$ into a topological space $X'$ and there exists a sheaf of $\mathbb{C}$-vector spaces $\mathcal{F}$ on $X'$ with $ H^{1}(X',\mathcal{F})=0$ and $\mathcal{F}\mid_X=\mathcal{O}_{X}$. If $X$ admits the Hartogs phenomenon, then $H^{1}_{c}(X,\mathcal{O}_X)=0$.
\end{corollary}

Let us note that for the implication "$\Leftarrow$" in Theorem \ref{th1} we could use the one-point compactification and the extension by zero of the sheaf $\mathcal{O}_{X}$. Namely, we have the following

\begin{corollary}\label{corth3}
Let $X$ be a noncompact normal complex analytic variety which is connected at boundary. If $\dim H^{1}_{c}(X,\mathcal{O}_{X})<\infty$, then $X$ admits the Hartogs phenomenon.
\end{corollary}

\begin{proof}
Let $X'$ be the one-point compactification of $X$ and $\mathcal{F}$ be the extension by zero of the sheaf $\mathcal{O}_{X}$. Since $H^{1}(X',\mathcal{F})= H^{1}_{c}(X',\mathcal{F})\cong H^{1}_{c}(X,\mathcal{O}_{X})$, it follows that $X$ admits the Hartogs phenomenon by Theorem \ref{th1}.
\end{proof}

For example, for $(1,\sigma)$-compactifiable complex manifolds we obtain the following

\begin{corollary}\label{corth2}
Let $X$ be a $(1,\sigma)$-compactifiable complex analytic manifold. Then $X$ admits the Hartogs phenomenon if and only if $\dim H^{1}_{c}(X,\mathcal{O}_X)<\infty$.
\end{corollary}

\begin{remark}
Let $X$ be a Stein manifold with $\dim X>1$. Since $X$ is connected at boundary \cite[Chapter VII, Section D, Theorem 2]{Ganning} and $H^{1}_{c}(X,\mathcal{O}_{X})=0$, it follows that $X$ admits the Hartogs phenomenon. 
\end{remark}

Let $X,Y,F$ be complex manifolds with countable basis. We obtain the following main result. 

\begin{theorem}\label{mainresult}
Let $\phi\colon X\to Y$ be a holomorphic fibre bundle with $(1,0)$-com\-pac\-ti\-fiable fiber $F$, $\dim F>1$. If $F$ admits the Hartogs phenomenon, then $X$ also admits the Hartogs phenomenon. 
\end{theorem}

\begin{proof}
If $F$ admits the Hartogs phenomenon, then $H^{1}_{c}(F,\mathcal{O}_{F})=0$ (by Corollary \ref{corth1}). By Corollary \ref{maincorcompactif} we obtain $H^{1}_{c}(X,\mathcal{O}_{X})=0$. So, by Corollary \ref{corth3} we obtain that $X$ admits the Hartogs phenomenon.

Of course, instead Corollary \ref{corth3} we could use the standard $\bar{\partial}$-technique.
\end{proof}
 
\begin{example}\label{Examtoric}
Let $G$ be a semiabelian Lie group (i.e. G is an extension of an abelian manifold $A$ by an algebraic torus $T\cong (\mathbb{C}^{*})^{n}$). We have a principle $T$-bundle $G\to A$. Consider a toric $T$-manifold $F$ (i.e. $T$ algebraically acts on $F$ with an open dense $T$-orbit which is isomorphic to $T$) which has only one topological end, and an associated fiber bundle $G\times^{T}F\to A$. 

Any toric $T$-manifold $F$ admits a $T$-equivariant compactification $F'$ and moreover $H^{1}(F',\mathcal{O}_{F'})=0$ (see, for instance, \cite{Oda}). It follows that any toric $T$-manifold $F$ which has only one topological end is $(1,0)$-compactifiable.  

It follows that if the toric $T$-manifold $F$ has only one topological end, $\dim F>1$ and $F$ admits the Hartogs phenomenon, then $G\times^{T}F$ admits the Hartogs phenomenon.

For example, $F=\mathcal{O}_{\mathbb{P}^{1}}(-2)$ has a structure of a toric $(\mathbb{C}^{*})^{2}$-action and it has only one topological end. Note that $\mathcal{O}_{\mathbb{P}^{1}}(-2)$ is isomorphic to the blow-up $Bl_{0}(Cone(Q))$ of the normal 2-dimensional affine cone $Cone(Q)\subset\mathbb{C}^{3}$ at point $(0,0,0)$, where $Q=\{[z_{0}:z_{1}:z_{2}]\in\mathbb{CP}^{2}\mid z_{1}z_{3}-z_{2}^{2}=0\}$. It follows that $F$ admits the Hartogs phenomenon, because blow-down map is a proper map and $Cone(Q)$ admits the Hartogs phenomenon (about Hartogs phenomenon in normal Stein spaces see, for instance, \cite{Vassiliadou}). 

Note that each toric variety $F$ is encoded by a fan $\Sigma_{F}$ — a collection of strictly convex cones in a real vector space $\mathbb{R}^{\dim F}$ with the common apex that may intersect only along their common faces (see, for instance, \cite{Oda}). In the fan language the (1,0)-compactifiability of $F$ means that the support $\mid\Sigma_{F}\mid$ of $\Sigma_{F}$ satisfies $\mathbb{R}^{\dim F}\setminus \mid\Sigma_{F}\mid$ is a connected set. Moreover, a toric manifold $F$ with a fan $\Sigma_{F}$ admits the Hartogs phenomenon if and only if the convex hull of $\mathbb{R}^{\dim F}\setminus \mid\Sigma_{F}\mid$ is whole $\mathbb{R}^{\dim F}$ provided  $\mathbb{R}^{\dim F}\setminus \mid\Sigma_{F}\mid$ is connected. About the Hartogs phenomenon in toric varieties (and, more general, in spherical varieties) and examples see in the papers \cite{Fek1, Fek2}.
\end{example}

Acknowledgments: This work is supported by the Krasnoyarsk Mathematical Center and financed by the Ministry of Science and Higher Education of the Russian Federation (Agreement No. 75-02-2023-936).

\end{document}